\documentclass[12pt,a4paper]{amsart}
\usepackage[centertags]{amsmath}
\usepackage{amsfonts}
\usepackage{amssymb}
\usepackage{amsthm}
\usepackage{newlfont}
\usepackage{color}

\newtheorem{theorem}{Theorem}[section]
\newtheorem{proposition}[theorem]{Proposition}
\newtheorem{lemma}[theorem]{Lemma}

\newcommand{\Tr}{\operatorname{Tr}}

\theoremstyle{definition}

\begin{document}

\title[Characterization of operator monotone functions]{On characterization of operator monotone functions}

\author[D. T. Hoa]{Dinh Trung Hoa$^b$} 

\address{Division of Computational Mathematics and Engineering (CME), Institute for Computational Science (INCOS), Ton Duc Thang University, Ho Chi Minh City, Vietnam; \\ Faculty of Civil Engineering, Ton Duc Thang University, Ho Chi Minh City, Vietnam;  \\ Department of Mathematics and Statistics, Auburn University, Auburn, Alabama, USA. 
}
\email{dinhtrunghoa@tdt.edu.vn; dzt0022@auburn.edu}


\keywords{characterization of operator monotonicity, mean of positive matrices, reverse Cauchy inequality} 
\subjclass[2000]{46L30, 15A45}

\footnote{$^B$This research is funded by Vietnam National Foundation for Science and Technology Development (NAFOSTED) under grant number 101.04-2014.40.}

\begin{abstract}
In this paper, we will show a new characterization of operator monotone functions by a matrix reverse Cauchy inequality.
\end{abstract}
\maketitle



\section{Introduction}
Let $M_n$ be the space of $n\times n$ complex matrices, $M_n^h$ the self-adjoint part of  $M_n$. For $A, B \in M_n^h$, the notation $A \le B$ means that $B - A \in  M_n^+$. The spectrum of a matrix $A\in M_n$ is denoted by $\sigma(A)$. For a real-valued function $f$ of a real variable and a matrix $A \in M_n^h$, the value $f(A)$ is understood by means of the functional calculus for Hermitian matrices.

Taking an axiomatic approach,  Kubo and Ando introduced the notions of  connection and mean. A binary operation $\sigma$ defined on the set of positive definite matrices is called a \textit{connection} if

(i) $A\leq C,B\leq D$ imply $A \sigma B \leq B\sigma D;$

(ii) $C^*(A\sigma B)C\leq (C^*AC)\sigma(C^*BC);$

(iii) $A_{n}\downarrow A$ and $B_{n}\downarrow B$ imply $A_{n}\sigma B_{n}\downarrow A\sigma B$.\\
If $I\sigma I=I,$ then $\sigma$ is called a \textit{mean}.

For $A, B>0$, the \textit{geometric mean} $A\sharp B$ is defined by
\begin{eqnarray*}
 A\sharp B &=& A^{1/2}(A^{-1/2}BA^{-1/2})^{1/2}A^{1/2}.
 \end{eqnarray*}
The \textit{harmonic} $A!B$ and \textit{arithmetic} $A\nabla B$ means are defined
$A!B = 2(A^{-1}+B^{-1})^{-1}$ and $A\nabla B=\displaystyle\frac{A+B}{2}$, respectively.  A mean $\sigma$ is called to be symmetric if $A\sigma B = B\sigma A$ for any pair of positive definite matrices $A, B$.

It is well-known that the arithmetic mean $\nabla$ is the biggest among symmetric means. From the general theory of symmetric matrix means we know that $\nabla\geq \sigma$ and $\tau\geq !$.

For positive real numbers $a, b$, the arithmetic-geometric mean inequality (AGM) says that
$$
\sqrt{ab} \le \frac{a+b}{2}.
$$
Hences, for a monotone increasing function $f$ on $[0, \infty)$, we have
\begin{equation}\label{cauchy}
f(\sqrt{ab}) \le f(\frac{a+b}{2}).
\end{equation}
It is natural to ask that if inequality (\ref{cauchy}) holds for any pair of positive number $a, b$ will the function $f$ be monotone increasing on  $[0, \infty)$? The answer is positive, and follows from the elementary fact that for any positive numbers $a \le b$ there exist positive number $x, y$ such that $a$ is arithmetic mean and $b$ is geometric mean of $x, y$.

The matrix version of above fact was investigated by Prof. T.Ando and Prof. F.Hiai \cite{Ando-Hiai}. They showed that the Cauchy inequality characterizes operator monotone functions, that means, if the following inequality holds
\begin{equation}\label{Hiai-ando}
f(A\nabla B) \ge f(A\sharp B)
\end{equation} 
whenever positive definite matrices $A, B$, then the function $f$ is an operator monotone.

An interesting and useful reverse Cauchy inequality is as follows: for any positive number $x, y$,
\begin{equation}\label{reverse cauchy-scalar1}
\frac{x+y}{2} \le \sqrt{xy}+ \frac{1}{2}|x-y|, 
\end{equation}
or 
$$
\min\{x, y\} \le \sqrt{xy}.
$$
The matrix case of (\ref{reverse cauchy-scalar1}) is the Powers-St\o rmer inequality which gives an upper bound of quantum Chernoff bound in quantum hypothesis testing theory \cite{K. M. R. Audenaert}. A generalized version of the Powers-St\o rmer inequality was studied by the author and H.Osaka and Ho M. Toan \cite{HOT}.

We have a scalar characterization of monotone functions as follows.

\begin{proposition}\label{Prop1} A function $f$ on $[0, \infty)$ is monotone increasing if and only if the following inequality  
\begin{equation}\label{reverse cauchy-scalar}
f(\frac{x+y}{2}) \le f(\sqrt{xy}+ \frac{1}{2}|x-y|) 
\end{equation}
holds for any pair of positive numbers $x, y$. 
\end{proposition}
\begin{proof}
Firstly, we show that for $0 \le a \le b \le \sqrt{2} a$ there exist positive numbers $x, y$ such that
$$
a = \frac{x+y}{2}, \quad b= \sqrt{xy}+ \frac{1}{2}|x-y|.
$$
We can assume that $x \le a$. It is easy to see that for $0 \le a \le b \le \sqrt{2} a$ the equation 
$$
b= \sqrt{x(2a-x)} + a-x,
$$
or, 
$$
2x^2 + 2(b-2a)x + (b-a)^2 = 0
$$
has a positive solution.
Consequently, if we have (\ref{reverse cauchy-scalar}), then 
$$
 f(a) = f(\frac{x+y}{2}) \le  f(\sqrt{xy}+ \frac{1}{2}|x-y|)  = f(b).
$$
For arbitrary $a\le b$, it is obvious that there exist numbers $a_1, a_2, \cdots, a_m$ such that
$$
a= a_0 \le a_1 \le \cdots \le a_m=b \quad \hbox{and} \quad a_i \le a_{i+1} \le \sqrt{2} a_i.
$$
Apply above arguments, we can get 
$$
f(a) \le f(a_1) \le f(a_2) \le \cdots \le f(a_n)= f(b).
$$
\end{proof}

In \cite{HOJ} it was shown a characterization of operator monotonicity by the trace Powers-St\o rmer's inequality in an infinite-dimensional Hilbert space. To continue this topic and partially motivated by above mentioned result of F.Hiai and T.Ando, in this paper we will study the matrix case of Proposition \ref{Prop1}. More precisely, it will be shown that for a nonnegative function $f$ on $[0, \infty)$ if
\begin{equation*}
f(A\nabla B) \le f(A\sharp B + \frac{1}{2} A^{1/2}|I - A^{-1/2}B A^{-1/2}|A^{1/2}) 
\end{equation*}
for any pair of positive definite matrices $A, B$, then $f$ is operator monotone on $[0, \infty).$

\section{Main results}

\begin{theorem}\label{A reverse Cauchy Ineq} 
Let $f$ be an operator monotone function on $[0,\infty)$, and $\sigma_f$ the operator mean (in sense of Kubo-Ando's theory) corresponding to $f$. Then for any pair of positive matrices $A, B$,
\begin{equation}\label{reverse Cauchy}
A \nabla B -A\sigma_f B \le  \frac{1}{2}A^{1/2}|I- A^{-1/2}BA^{-1/2}|A^{1/2},
\end{equation}
where $\nabla$ is the arithmetic mean. 
\end{theorem}
\begin{proof}
Since $f$ is an operator monotone function on $[0, \infty)$, then it is well-known that 
\begin{equation*}\label{scalar2}
\min\{1, t\}  \le f(t) \quad (t \ge 0)
\end{equation*}
or 
\begin{equation*}\label{scalar3}
\min\{a, b\}  \le af(b/a) \quad (a, b > 0)
\end{equation*}

Hence, by the Spectral Theorem, for any positive matrix $Q$,
\begin{equation}\label{2}
I+ Q - |I-Q| \le 2 f(Q).
\end{equation}
For any positive matrices $B, C$, we substitute $Q=C^{*-1}BC^{-1}$ into (\ref{2}), and after multiplying on $C^*$ on the left hand side and $C$ on the right hand side, we get
\begin{equation}
C^* C + B - C^*|I- C^{*-1}BC^{-1}|C \le 2 C^*f(C^{*-1}BC^{-1}) C.
\end{equation}
Since $A$ is positive, we can take $A=C^{1/2}$, then we get
\begin{equation}
A + B - A^{1/2}|I- A^{-1/2}BA^{-1/2}|A^{1/2} \le 2 A\sigma_f B,
\end{equation}
or 
\begin{equation*}
A \nabla B -A\sigma_f B \le  \frac{1}{2}A^{1/2}|I- A^{-1/2}BA^{-1/2}|A^{1/2}.
\end{equation*}
\end{proof}

In \cite{Hoa-Osaka-Khue} the author with H.Osaka and Vo T.B. Khue obtained the following theorem.

\begin{theorem}[\cite{Hoa-Osaka-Khue}]\label{thm:cauchy inequality}
Let $f$ be a strictly positive operator monotone function on $[0, \infty)$ with $f((0, \infty)) \subset (0, \infty)$
and $f(1) = 1$. Then
\begin{equation}\label{HOK}
A \nabla B -A \sigma_f B \le \frac{1}{2}|A - B|
\end{equation}
for any positive semidefinite matrices $A$ and $B$ satisfying the condition
\begin{equation}\label{condition}
AB + BA \geq 0.
\end{equation}
\end{theorem}
We also give a counterexample for inequality (\ref{HOK}) if the condition (\ref{condition}) does not hold. If the mean $\sigma_f$ in Theorem \ref{thm:cauchy inequality} is symmetric, then without the condition (\ref{condition}) inequality (\ref{HOK}) holds for any positive matrices $A, B$ iff $\sigma_f$ is arithmetic. The proof of the following theorem is adapted from \cite{Ando-Hiai}.

\begin{theorem}
Let $\sigma$ be a symmetric mean, and let the following inequality 
\begin{equation}\label{HOK2}
A\nabla B - A\sigma B \le \frac{1}{2}|A-B|
\end{equation}
holds for any pair of positive semidefinite matrices $A, B$. Then $\sigma = \nabla$.
\end{theorem}
\begin{proof} By \cite[Theorem 4.4]{Kubo-Ando}, the symmetric operator mean $\sigma$ is represented for any positive operators $A, B$ as
\begin{equation}\label{Kubo-Ando-4.4}
A \sigma B = \frac{\alpha}{2} (A +B) + \int_{(0,\infty)} \frac{\lambda +1}{\lambda} \{(\lambda A):B + A :(\lambda B)\} d \mu(\lambda), 
\end{equation}
where $\lambda \ge 0$ and $\mu$ is a positive measure on $(0, \infty)$ with $\alpha + \mu((0, \infty))=1$. 
If $P, Q$ are orthogonal projections such that $P \wedge Q = 0$ then by \cite[Theorem 3.7]{Kubo-Ando}, 
$$(\lambda P):Q = P : (\lambda Q) = \frac{\lambda}{\lambda + 1} P \wedge Q.$$ 
Consequently, from (\ref{Kubo-Ando-4.4}) we get
 $$P \sigma Q = \frac{\alpha}{2} (P+Q).$$
So, the inequality (\ref{HOK2}) becomes
\begin{equation}\label{4}
(1-\alpha) (P+Q) \le |P-Q|.
\end{equation}
Let us consider the following orthogonal projections (see \cite{Ando-Hiai})
$$
P = \left(
\begin{array}{cc}
1&0\\
0&0
\end{array}
\right), \quad Q = \left(
\begin{array}{cc}
\cos^2 \theta & \cos \theta \sin \theta \\
\cos \theta \sin \theta &\sin^2 \theta
\end{array}
\right).
$$
For such operators, let's compare the (1, 1)-entries of both sides of (\ref{4}) we get
$$
(1-\alpha) (1+\cos^2 \theta) \le \sin^2 \theta,
$$
or equivalent,
$$
1-\alpha \le \frac{\sin^2 \theta}{1+\cos^2 \theta}.
$$
Tending $\theta$ to zero, we get that $\alpha \ge 1$. This shows that $\mu =0$ in (\ref{Kubo-Ando-4.4}), and hence $\sigma =\nabla$.
\end{proof}

Now, we will prove the main result of this paper, that is to prove that inequality (\ref{A reverse Cauchy Ineq}) characterizes the operator monotonicity. 

We need two lemmas.
\begin{lemma}\label{lemma1}
Let $I \le A \le \sqrt{2}I$. Then there exist positive matrices $X, Y$ such that 
$$
I = X \nabla Y,\quad A = X\sharp Y + \frac{1}{2} X^{1/2}|I -X^{-1/2} YX^{-1/2}| X^{1/2}.
$$
\end{lemma}
\begin{proof}
To prove Lemma, we need to show that there exists a positive matrix $X$ such that
$$
A = X\sharp (2I-X) + \frac{1}{2} X^{1/2}|I -X^{-1/2} (2I-X)X^{-1/2}| X^{1/2},
$$
or, equivalent,
$$
A = X\sharp (2I-X) + |X-I| 
$$
for the given positive matrix $A$. We can assume that $X \le I.$ And so, we should show that for any positive matrix $I \le A \le I \sqrt{2}$ we can find a matrix $0\le X \le I$ such that 
$$
A = X\sharp (2I-X) + I - X.
$$
It is easy to see that the function $h(t) = \sqrt{2t - t^2} + 1-t$ will be strictly increasing on $[0, \frac{2-\sqrt{2}}{2}]$ and take values in interval $[1, \sqrt{2}]$. So, for a such matrix $A$, we can find a matrix $X$ such that $X = h^{-1}(A)$.
\end{proof}
\begin{lemma}\label{lemma2}
Let $A \le B \le \sqrt{2}A$. Then there exist positive matrices $X, Y$ such that 
$$
A = X \nabla Y,\quad B = X\sharp Y + \frac{1}{2} X^{1/2}|I -X^{-1/2} YX^{-1/2}| X^{1/2}
$$
\end{lemma}
\begin{proof}
From the assumption we have $I \le A^{-1/2}B A^{-1/2} \le \sqrt{2} I$. From the proof of Lemma \ref{lemma1},  there exist positive matrices $X_0, Y_0$ 
$$
A^{-1/2} B A^{-1/2} = X_0 \sharp Y_0 + I - X_0, \quad X_0 \nabla Y_0 = I.
$$
Consequently,
\begin{equation}\label{12}
B = (A^{1/2}X_0A^{1/2}) \sharp (A^{1/2}Y_0A^{1/2})  + A - A^{1/2}X_0A^{1/2}.
\end{equation}
In the other hand, from the proof of Lemma \ref{lemma1} we have that $$X_0 \le I \le Y_0 \quad \hbox{and} \quad X_0 + Y_0 = 2I.$$ Hence 
\begin{equation*}
\begin{split}
&A - A^{1/2}X_0A^{1/2} \\
& = A^{1/2}(I - X_0)A^{1/2} \\
&= \frac{1}{2}A^{1/2}(Y_0 - X_0)A^{1/2} \\
&=\frac{1}{2} (A^{1/2}Y_0A^{1/2} - A^{1/2}X_0A^{1/2}) \\
& = \frac{1}{2} (A^{1/2}X_0A^{1/2})^{1/2}((A^{1/2}X_0A^{1/2})^{-1/2}(A^{1/2}Y_0A^{1/2})(A^{1/2}X_0A^{1/2})^{-1/2} -I)  \\ 
& \cdot (A^{1/2}X_0A^{1/2})^{1/2}\\
&=  \frac{1}{2} (A^{1/2}X_0A^{1/2})^{1/2}|(A^{1/2}X_0A^{1/2})^{-1/2}(A^{1/2}Y_0A^{1/2})(A^{1/2}X_0A^{1/2})^{-1/2} -I| \\
&\cdot (A^{1/2}X_0A^{1/2})^{1/2}.
\end{split}
\end{equation*}
\end{proof}

\begin{theorem}
 Let $f$ be a nonnegative function on $[0, \infty)$, and let the following inequality
\begin{equation*}
f(X \nabla Y) \le   f(X\sharp Y + \frac{1}{2} X^{1/2}|I -X^{-1/2} YX^{-1/2}| X^{1/2})
\end{equation*}
holds whenever positive definite matrices $X, Y$. Then the function $f$ is operator monotone on $[0, \infty).$
\end{theorem}
\begin{proof}
Let $0 \le A \le B \le A \sqrt{2}.$ From Lemma \ref{lemma2}, there exist positive matrices $X, Y$ such that  
$$
A = X \nabla Y,\quad B = X\sharp Y + \frac{1}{2} X^{1/2}|I -X^{-1/2} YX^{-1/2}| X^{1/2}. 
$$
On account of the assumption, we have
$$
f(A) = f(X \nabla Y) \le f(X\sharp Y + \frac{1}{2} X^{1/2}|I -X^{-1/2} YX^{-1/2}| X^{1/2}) = f(B). 
$$
In general, let $A \le B$, hence $I \le A^{-1/2}BA^{-1/2}.$ It is obvious that we can find matrices $I=Z_1 \le Z_2 \le \cdots \le Z_k = A^{-1/2}BA^{-1/2}$ such that 
$$Z_i \le Z_{i+1} \le Z_i \sqrt{2}, \quad i=0, \cdots, n-1.$$ Consequently,
$$
0 \le A \le A^{1/2}Z_2 A^{1/2} \le \cdots \le A^{1/2}Z_n A^{1/2} = B
$$
and 
$$
A^{1/2}Z_i A^{1/2}  \le A^{1/2}Z_{i+1} A^{1/2} \le \sqrt{2}A^{1/2}Z_i A^{1/2}.
$$
From the above arguments, we have

$$f(A) \le f(A^{1/2}Z_2A^{1/2}) \le \cdots \le f(A^{1/2}Z_{n} A^{1/2}) =f(B).$$
That means, the function $f$ is operator monotone.
\end{proof}

{\it Acknowledgement.} The author would like to thank Professor Fumio Hiai for useful comments on this work. This research is funded by Vietnam National Foundation for Science and Technology Development (NAFOSTED) under grant number 101.04-2014.40.

\end{document}